\documentclass{birkjour}

\usepackage{bm}

 \usepackage{upgreek}

\usepackage{amsopn, amsthm, amsgen, amscd,amsmath, amssymb}

\DeclareMathAlphabet{\mathbbm}{U}{bbm}{m}{n}
\usepackage{color}
\usepackage{tikz}
  \usepackage{graphicx}
\usepackage{epstopdf}
\usepackage[small,nohug,heads = LaTeX]{diagrams} \diagramstyle[labelstyle=\scriptstyle]

\definecolor{CadetBlue}{cmyk}{0.62, 0.57, 0.23, 0 }
\definecolor{black}{cmyk}{1, 0.5, 0, 0 }
\definecolor{RedViolet}{cmyk}{0.07, 0.9, 0, 0.34 }
\definecolor{SeaGreen}{cmyk}{0.69, 0, 0.5, 0}

\DeclareMathAlphabet{\mathpzc}{OT1}{pzc}{m}{it}

\newcommand{\R}{\mathbb R}
\newcommand{\C}{\mathbb C}

\newcommand{\F}{\mathbb F}

\newcommand{\N}{\mathbb N}

\newcommand{\Q}{\mathbb Q}
\newcommand{\Z}{\mathbb Z}

\newcommand{\e}{\upvarepsilon}

\newtheorem{theo}{Theorem}
\newtheorem{lemm}{Lemma}
\newtheorem{prop}{Proposition}
\newtheorem{coro}{Corollary}

\theoremstyle{definition}

\theoremstyle{remark}

\newtheorem{note}{Note}

\title[Quantum $j$-Invariant in Positive Characteristic I]{Quantum $j$-Invariant in Positive Characteristic I: Definition and Convergence}
\author{L. Demangos}
\address{Instituto de Matem\'{a}ticas -- Unidad Cuernavaca, Universidad
Nacional Autonoma de M\'{e}xico, Av. Universidad S/N, C.P. 62210
Cuernavaca, Morelos, M\'{e}xico}
\curraddr{Department of Mathematical Sciences (Mathematics Division),
University of Stellenbosch,
Private Bag X1,
Matieland 7602,
South Africa}
\email{l.demangos@gmail.com}
\author{T.M. Gendron}
\address{Instituto de Matem\'{a}ticas -- Unidad Cuernavaca, Universidad
Nacional Autonoma de M\'{e}xico, Av. Universidad S/N, C.P. 62210
Cuernavaca, Morelos, M\'{e}xico}
\email{tim@matcuer.unam.mx}
\begin{document}
\vspace{2cm}

 \maketitle
 
 \begin{abstract}  
We introduce the quantum $j$-invariant in positive characteristic as a multi-valued, modular-invariant function of a local function field.  In this paper, we
  concentrate on basic definitions and questions of convergence.
  \end{abstract}

\section{Introduction}

In \cite{Ge-C}, a notion of modular invariant for quantum tori was introduced, which may be viewed analytically as a certain extension of the usual $j$-invariant to the real numbers.  
For any $\uptheta\in \R$ and $\upvarepsilon >0$, one defines first
the approximant $j_{\upvarepsilon}(\uptheta )$, using Eisenstein-like series over the set of ``$\upvarepsilon$ diophantine approximations''
\[ \Uplambda_{\upvarepsilon}(\uptheta ) =\{ n\in\Z|\; |n\uptheta-m|<\upvarepsilon \text{ for some }m\in\Z \}.\]
Then $j^{\rm qt}(\uptheta )$ is defined
to be the set of limits of the approximants as $\upvarepsilon\rightarrow 0$.

PARI-GP experiments indicate that $j^{\rm qt}(\uptheta )$ is multi-valued and yet remarkably, for all quadratics tested, the set of all approximants
is finite (see the Appendix in \cite{Ge-C}).  However due to the chaotic nature of the sets $\Uplambda_{\upvarepsilon}(\uptheta ) $,
explicit expressions for $j^{\rm qt}(\uptheta )$
have proven to be elusive: in \cite{Ge-C}, only a single value of $j^{\rm qt}(\upvarphi )$, where $\upvarphi=$ the golden mean,
was rigorously computed.  Moreover
the experimental observations made at the quadratics remain without proof.

This paper is the first in a series of two, in which we consider the analog of $j^{\rm qt}$ for a function field over a finite field.  In this case, the non archimedean nature of
the absolute value simplifies the analysis considerably, allowing us to go well beyond what was obtained in \cite{Ge-C}.  In particular,
for $f$ belonging to the function field analog of the reals,
the set $\Uplambda_{\upvarepsilon}(f) $ has the structure of an $\F_{q}$-vector space and it is possible to describe a basis for it
using the sequence of denominators of best approximations of $f$.  With this in hand, explicit formulas for the approximants $j_{\upvarepsilon}(f)$ and their
absolute values can be given, see Theorem \ref{convergencetheo}  and its proof.  Using these formulas, we are able to prove the multi-valued property of $j^{\rm qt}$, as well as give the characterization: $f$ is rational if and only if $j_{\upvarepsilon}(f)=\infty$ eventually.
In the sequel \cite{DGII}, we study the set of values of $j^{\rm qt}(f)$ for $f$ quadratic.

\noindent {\it Acknowledgements.} We thank E.-U. Gekeler as well as the referee for a number of useful suggestions.

\section{Quantum $j$ invariant}\label{QJInv}

Let $\F_{q}$ be the field with $q=p^{r}$ elements, $p$ a prime.  Let 
$k=\F_{q}(T)$ where $T$ is an indeterminate and $A=\F_{q}[T]$.
Let $v$ be the place at $\infty$ i.e.\ that defined by
$v(f)=-\deg_{T}(f) $.   
The absolute value $|\cdot |$
associated to $v$ is \[ |f | = q^{-v(f)}=q^{\deg_{T}(f)}.\]
Denote by $k_{\infty}$ the local field obtained as the completion of $k$ w.r.t.\ $v$, and
by ${\bf C}_{\infty}$ the completion of a fixed algebraic closure $\overline{k_{\infty}}$ of $k_{\infty}$ w.r.t.\ $v$. For a review of basic function field arithmetic, see \cite{Goss}, \cite{Thak}, \cite{VS}, \cite{Weil}.

Consider the compact $A$-module \[ {\bf S}^{1}:=  k_{\infty}/A. \]   The character group of ${\bf S}^{1}$ may
be identified with the group of characters $e:k_{\infty}\rightarrow \C^{\times}$ trivial on $A$.   Define the basic character 
$e_{0}(g) =\upchi (c_{-1}(g))$, where $\upchi :(\F_{q},+)\rightarrow \C^{\times}$
is a nontrivial character
with values in the group of $p$th roots of unity and $c_{-1}(g)$ is the coefficient of $T^{-1}$ in the Laurent series expansion of $g$.   Then $e_{0}$ induces a character
of ${\bf S}^{1}$ and moreover every character of ${\bf S}^{1}$ is of the form $e(g)=e_{0}(ag)$ for some $a\in A$.  Indeed, every character
of $k_{\infty}$ may be written in the form $e(g)=e_{0}(rg)$ for some $r\in k_{\infty}$ {\it c.f.}\
 \cite{Weil}, sect. II.5.  If $e$ is trivial on $A$ but $r\not\in A$, we could find an element $a\in A$ with $c_{-1}(ar)\not\in {\rm Ker}(\upchi )$.
 But then $e(a)\not=1$, contradiction.

We will need the following analog of the classical Kronecker's Theorem (see \cite{Cas}).

\begin{theo}\label{kron}  For any $f\in k_{\infty}-k$, the image of $Af$ in ${\bf S}^{1}$ is dense.
\end{theo}

\begin{proof}  For any compact Hausdorff abelian group $G$, the Weyl Criterion ({\it c.f.}\ \cite{KN}, Corollary 1.2 of Chapter 4) is available. In particular, the 
$A$-sequence $x_{a}:= af\mod A$, $a\in A$, is uniformly distributed in ${\bf S}^{1}$ if and only if 
\begin{align}
\lim_{d\rightarrow\infty} \frac{1}{q^{d+1}} \sum_{\deg a\leq d} e(x_{a}) = 0 
\end{align}
for each non-trivial character $e: {\bf S}^{1}\rightarrow \C^{\times}$.  Since $f$ is irrational, for each non-trivial character $e$ there exists $b\in A$ such that
$e(x_{b})\not =1$.  Indeed, 
 if it were true
that for some non-trivial character $e$ and for all $b\in A$,
\[ e(x_{b})=e_{0}(ax_{b})=e_{0}(x_{ab})=e_{0}(bx_{a})=1,\] this would imply $af\in A$ i.e. $f\in k$, contradiction.
If $\updelta=\deg b$, then for $d>\updelta$ we have
\begin{align} (1-e(x_{b}))  \left( \sum_{\deg a\leq d} e(x_{a}) \right) & =\sum_{\deg a\leq d} e(x_{a})  - \sum_{\deg a\leq d} e(x_{a+b})  \label{firstsum}\\
& =  (1-e(x_{b}))  \left( \sum_{\deg a\leq \updelta} e(x_{a}) \right), \label{2ndsum}
\end{align}
since the map $a\mapsto a+b$ defines a bijection of the set of $a$ satisfying $\updelta< \deg (a)\leq d$, i.e.,
\[ \sum_{\updelta<\deg a\leq d} e(x_{a})  - \sum_{\updelta< \deg a\leq d} e(x_{a+b})=0  .\] 
As $1-e(x_{b})\not=0$, it may be canceled in (\ref{firstsum}),  (\ref{2ndsum}) to give equality of the two summations with limits $d$, $\updelta$, i.e., 
\[\lim_{d\rightarrow\infty} \frac{1}{q^{d+1}}  \sum_{\deg a\leq d} e(x_{a})  =\lim_{d\rightarrow\infty} \frac{1}{q^{d+1}}  \sum_{\deg a\leq \updelta} e(x_{a}) =0 . \]
 Thus $\{ x_{a}\}$ is uniformly distributed
in ${\bf S}^{1}$ and so
dense in ${\bf S}^{1}$.
\end{proof}

By a {\it lattice} $\Uplambda\subset {\bf C}_{\infty}$ is meant a discrete $A$-submodule of finite rank.  In what follows, we restrict to lattices
of rank 2 e.g. 
$\upomega\in\Upomega = {\bf C}_{\infty}- k_{\infty}$ defines the rank 2 lattice
$\Uplambda (\upomega) = \langle 1, \upomega\rangle_{A}=$ the $A$-module generated by $1,\upomega$.
Given a lattice $\Uplambda\subset  {\bf C}_{\infty}$, the {\it Eisenstein series} are
\[  E_{n}(\Uplambda) = \sum_{0\not=\uplambda\in\Uplambda} \uplambda^{-n},\quad n\in\N.\]
The {\it discriminant} is the expression
\[  \Updelta (\Uplambda) := (T^{q^{2}}-T)E_{q^{2}-1} (\Uplambda)+ (T^{q}-T)^{q}E_{q-1} (\Uplambda)^{q+1}\]
and setting
\[  g (\Uplambda) := (T^{q}-T)E_{q-1} (\Uplambda),\]
the (classical) {\it $j$ invariant} is defined
\[ j (\Uplambda) :=\frac{g (\Uplambda)^{q+1}}{\Updelta (\Uplambda)} = \frac{1}{\frac{1}{T^{q}-T}  + J  (\Uplambda)} \]
where
\[ J  (\Uplambda):=  \frac{T^{q^{2}}-T}{(T^{q}-T)^{q+1}} \cdot
\frac{E_{q^{2}-1} (\Uplambda)}{E_{q-1} (\Uplambda)^{q+1}}.
\]
When we restrict to $\Uplambda = \Uplambda (\upomega)$ we obtain a well-defined modular function 
\[  j: {\rm PGL}_{2}(A)\backslash\Upomega\longrightarrow  {\bf C}_{\infty}.\] 
See for example \cite{Gek}, \cite{Gek1}.

Let $f\in k_{\infty}$.  The $A$-module $\langle 1, f\rangle_{A}\subset k_{\infty}$ is a lattice only when $f\in k$; when $f\in k_{\infty}-k$, Theorem \ref{kron}
implies that $\langle 1, f\rangle_{A}$ is dense in $k_{\infty}$.  In order to define $j^{\rm qt}(f)$, we look for a suitable replacement for the lattice used to define $j(\upomega)$ for $\upomega\in\Upomega$.  We adapt the ideas found in \cite{Ge-C}.

For any $x\in k_{\infty}$ we denote the distance to the nearest element of $A$ by 
\[  \| x\| := \min_{a\in A}  |x-a| \in [0,1[.\]
Given $\upvarepsilon>0$, define\[\Uplambda_{\upvarepsilon}(f):=\{\uplambda\in A|\;  \| \uplambda f\|<\upvarepsilon\}.\]
Elements $\uplambda\in \Uplambda_{\upvarepsilon}(f)$ are referred to as {\it $\upvarepsilon$-approximations} of $f$ and the value $ \| \uplambda f\|$
is called the {\it error} of $\uplambda$.

Since the distance between distinct elements of $A$ is $\geq 1$, if we take $\upvarepsilon <1$, then for all $\uplambda\in \Uplambda_{\upvarepsilon}(f)$ there exists a unique $\uplambda^{\perp}\in A$ such that
\[ \| \uplambda f\| =| \uplambda f-\uplambda^{\perp}| <\upvarepsilon.\]
We call $\uplambda^{\perp}$ the {\it $f$-dual} of $\uplambda$, and we denote $\Uplambda_{\upvarepsilon}(f)^{\perp}=\{\uplambda^{\perp}|\; \uplambda\in \Uplambda_{\upvarepsilon}(f)\}$.

\begin{prop}\label{groupideal} For each $0<\upvarepsilon <1$, $\Uplambda_{\upvarepsilon}(f)$, $\Uplambda_{\upvarepsilon}(f)^{\perp}\subset A$ are isomorphic 
$\F_{q}$-vector spaces.  Moreover, they are $A$-ideals for $\e$ sufficiently small if and only if $f\in k$. 
\end{prop}

\begin{proof}
Given $\uplambda,\uplambda'\in \Uplambda_{\upvarepsilon}(f)$,  
$ |(\uplambda+\uplambda')f-(\uplambda^{\perp}+\uplambda'^{\perp})|  =|(\uplambda f-\uplambda^{\perp})+(\uplambda'f-\uplambda'^{\perp})| 
                   \leq \max\{|\uplambda f-\uplambda^{\perp}|,|\uplambda'f-\uplambda'^{\perp}| \} 
                   <\upvarepsilon$, 
                   which shows that $\Uplambda_{\upvarepsilon}(f)$ is a group; in fact a vector space since multiplication by scalars leaves $|\cdot |$ invariant. The map $\uplambda\mapsto \uplambda^{\perp}$ defines a bijection respecting vector space operations, making $\Uplambda_{\upvarepsilon}(f)^{\perp}$
                  a vector space isomorphic to $\Uplambda_{\upvarepsilon}(f)$.  This takes care of the first statement. Concerning the second statement,  suppose that $f\in k$ and write $f=a/b$, where $a,b\in A$ are relatively prime. Then
                  for $\upvarepsilon\leq |b|^{-1}$, the inequality $|\uplambda (a/b)-\uplambda^{\perp}|<\upvarepsilon$ implies that $\uplambda a - \uplambda^{\perp}b=0$ or $\uplambda^{\perp} =  \uplambda a/b$.  Since
                  $(a,b)=1$, this implies that $b|\uplambda$
so that                   $\Uplambda_{\upvarepsilon}(f) \subset bA$.   The opposite inclusion $\Uplambda_{\upvarepsilon}(f) \supset bA$ is trivial.  
On the other hand, 
                  suppose $f\in k_{\infty}- k$ and let $\uplambda\in \Uplambda_{\upvarepsilon}(f)$.
                  By Theorem 1 for each $c\in k_{\infty}$ such that $|c|<1$ there exists $\uplambda' \in A$ such that $|\uplambda \uplambda' f-c|<\e$. If we assume $\uplambda \uplambda'\in \Uplambda_{\upvarepsilon}(f)$ there exists $(\uplambda \uplambda')^{\perp}\in A$ such that $|\uplambda\uplambda'f-(\uplambda\uplambda')^{\perp}|<\e$. However:\begin{align}\label{maxineq} |\uplambda\uplambda'f-(\uplambda\uplambda')^{\perp}| & \leq \max\{|\uplambda\uplambda'f-c|,|(\uplambda\uplambda')^{\perp}-c|\}.\end{align}As $|c|<1$, we have
                  \[     |(\uplambda\uplambda')^{\perp}-c|  \geq 1 > |\uplambda \uplambda' f-c|.\]
                   Thus the inequality in (\ref{maxineq}) is an equality, i.e.\ \[  \upvarepsilon> |\uplambda\uplambda'f-(\uplambda\uplambda')^{\perp}|=|(\uplambda\uplambda')^{\perp}-c|\geq 1,\]contradiction.  In particular, $\Uplambda_{\upvarepsilon}(f)$ is not an ideal for all $\upvarepsilon\in (0,1)$.  Therefore, $\Uplambda_{\upvarepsilon}(f)$ is an $A$-ideal for $\upvarepsilon$ sufficiently small if and only if $f\in k$.
\end{proof}

\begin{note} In the number field setting, the analogous set $\Uplambda_{\upvarepsilon}(\uptheta)$ for $\uptheta\in\R-\Q$  is {\it not} a group,
due to the archimedean nature of the absolute value
of $\R$, see \cite{Ge-C}.
\end{note}

\begin{note} If $f\in k_{\infty}-k$ then the minimal degree of the nonzero elements of $\Uplambda_{\upvarepsilon}(f)$  tends to infinity as $\e\rightarrow 0$ because there are only finitely many polynomials in $T$ of fixed degree.  In other words, 
\[  \lim_{\e\rightarrow 0} \min \{ |\uplambda|\, |\; \uplambda\in \Uplambda_{\e}(f)-0\} =\infty. \]
Thus 
\[ \bigcap_{\upvarepsilon>0}\Uplambda_{\upvarepsilon}(f)=\{ 0\}\] and we could not use the intersection to
define a lattice-like object with which one might hope to define the quantum $j$-invariant.
\end{note}

The idea now is to use $\Uplambda_{\upvarepsilon}(f)$ as if it were a lattice to define an approximation to $j^{\rm qt}(f)$.   We 
define the {\it $\upvarepsilon$-zeta function} of $f$ via:\[ \upzeta_{f,\upvarepsilon}(n):=\sum_{\uplambda\in \Uplambda_{\upvarepsilon}(f)-\{0\}\atop
\uplambda\text{ monic }}\uplambda^{-n}, \quad n\in \N . \]
It is easy to see that $ \upzeta_{f,\upvarepsilon}(n)$ converges, since it is a subsum of  
the {\it zeta function} of $A$ (see \cite{Thak}) $ \upzeta_{A}(n) = \sum_{a\in A\text{ monic}}a^{-n}$.  If we denote 
$\zeta_{\F_{q}} (n)= \sum_{c\in \F_{q}-\{ 0\}} c^{-n}$
then
\[ \zeta_{\F_{q}} (n)\cdot  \upzeta_{f,\upvarepsilon}(n) = \sum_{\uplambda\in \Uplambda_{\upvarepsilon}(f)-\{0\}}\uplambda^{-n}.  \]
Taking into account that $ \zeta_{\F_{q}} (q-1)= \zeta_{\F_{q}} (q^{2}-1)=-1$, we
define
\begin{align*} \Updelta_{\upvarepsilon} (f) & := 
(T^{q^{2}}-T)\sum_{\uplambda\in \Uplambda_{\upvarepsilon}(f)-\{0\}}\uplambda^{1-q^{2}}+ (T^{q}-T)^{q}\left(\sum_{\uplambda\in \Uplambda_{\upvarepsilon}(f)-\{0\}}\uplambda^{1-q}\right)^{q+1}
 \\
 & =-(T^{q^{2}}-T)\upzeta_{f,\upvarepsilon}(q^{2}-1)+ (T^{q}-T)^{q}\upzeta_{f,\upvarepsilon}(q-1)^{q+1}
 \end{align*} and 
\[g_{\upvarepsilon} (f) := (T^{q}-T) \sum_{\uplambda\in \Uplambda_{\upvarepsilon}(f)-\{0\}}\uplambda^{1-q}=-(T^{q}-T)\upzeta_{f,\upvarepsilon}(q-1).\]
Then the {\it $\upvarepsilon$-modular invariant} of $f$ is defined
\[ j_{\upvarepsilon}(f):=\frac{{g_{\upvarepsilon}}^{q+1}(f)}{\Updelta_{\upvarepsilon}(f)} = \frac{1}{\frac{1}{T^{q}-T}  -J_{\upvarepsilon}(f) } \]
where
\begin{align}\label{Jep} J_{\upvarepsilon}(f) :=  \frac{T^{q^{2}}-T}{(T^{q}-T)^{q+1}} \cdot
\frac{\upzeta_{f,\e}(q^{2}-1)}{\upzeta_{f,\e}(q-1)^{q+1}}.
\end{align}

The {\it quantum modular invariant} or {\it quantum j-invariant}  of $f$ is then 
\[  j^{\rm qt}(f) := \lim_{\e\rightarrow 0} j_{\upvarepsilon}(f)   \subset k_{\infty}\cup \{ \infty \}\]
where by $ \lim_{\e\rightarrow 0} j_{\upvarepsilon}(f) $ we mean the {\it set of limit points} of convergent sequences $\{ j_{\e_{i}}(f)\}$, $\e_{i}\rightarrow 0$.
Thus, by definition, $j^{\rm qt}(f)$ is in principle multi-valued, and so we denote it as such, writing
\[ j^{\rm qt}: (k_{\infty}\cup \{ \infty\})\multimap (k_{\infty}\cup \{ \infty\}),\quad j^{\rm qt}(\infty ):=\infty.\]
In the number field case, $j^{\rm qt}$ is only known experimentally to be multi-valued, having finitely many values for all real quadratics tested, see the Appendix
in \cite{Ge-C}.   In this paper, we will prove that $j^{\rm qt}$ is multi-valued and in the sequel \cite{DGII}, we will prove
that $\# j^{\rm qt}(f)<\infty$ for $f$ quadratic.

\begin{theo}\label{rationaltheorem}
If $f\in k$, then for $\upvarepsilon$ sufficiently small, $j_{\upvarepsilon}(f)=\infty$.  In particular, $j^{\rm qt}(f)=\infty$.
\end{theo}
\begin{proof}
By Proposition \ref{groupideal}, $\Uplambda_{\upvarepsilon}(f)$ is an ideal for sufficiently small $\e$, hence it is a principal ideal of the form $(h)$, $h\in A$.  
Therefore, the quotient of $\e$ zeta functions occurring in (\ref{Jep}) in the definition of $J_{\e}(f)$, being homogeneous of degree $1-q^{2}$, can be written
\begin{align*}  \frac{\upzeta_{f,\e}(q^{2}-1)}{\upzeta_{f,\e}(q-1)^{q+1}} & =-\frac{\sum_{0\not=\uplambda\in \Uplambda_{\e}(f)}\uplambda^{-q^{2}+1}}{\left(\sum_{0\not=\uplambda\in \Uplambda_{\e}(f)}\uplambda^{1-q}\right)^{q+1}} 
\nonumber \\
& =-\frac{h^{q^{2}-1}\sum_{0\not=a\in A} a^{-q^{2}+1}}{h^{q^{2}-1}\left(\sum_{0\not=a\in A}a^{1-q}\right)^{q+1}} \nonumber \\
 & = \frac{\upzeta_{A}(q^{2}-1)}{\upzeta_{A}(q-1)^{q+1}}.
\end{align*} 
This shows that $j_{\e}(f)$ is constant for small $\e$, hence $j^{\rm qt}(f)$ is a single point.  Moreover, since $A$ is the limit of the lattice $\langle 1,\upomega\rangle_{A}$, $\upomega\in\Upomega$, for $\upomega\rightarrow 0$, the above calculation shows that
for $\e$ small, $j_{\upvarepsilon}(f)=j^{\rm qt}(f)$ is the limit of classical $j$ at $\infty$.  But it is well-known that
classical $j$ has a pole at $\infty$, see \cite{Gek1}, .
\end{proof}
 \begin{prop} $j^{\rm qt}$ is invariant w.r.t.\ the action of ${\rm PGL}_{2}(A)$ and induces a multi-valued function
 \[  j^{\rm qt}: {\rm PGL}_{2}(A)\backslash(k_{\infty}\cup \{ \infty\})\multimap (k_{\infty}\cup \{ \infty\}).\]
 \end{prop}
 
 \begin{proof} Essentially the same as the proof of Theorem 1 in \cite{Ge-C}. By Theorem \ref{rationaltheorem}, it is enough to
 prove ${\rm PGL}_{2}(A)$-invariance for $f\in k_{\infty}-k$. In brief,  if $M=\left(  \begin{array}{cc}
 r & s \\
 t & u
 \end{array}\right)\in {\rm GL}_{2}(A)$, then the map 
 $ \uplambda\longmapsto t\uplambda^{\perp} + u\uplambda $ induces an isomorphism of $\F_{q}$-vector spaces 
 \[  M: \Uplambda_{\upvarepsilon}(f)\longrightarrow  \Uplambda_{\upvarepsilon'}(M(f)) ,\quad \upvarepsilon' = \frac{\upvarepsilon}{ tf+u}.\]
Then one may check that $\{ \upvarepsilon_{i}\}$ produces a limit point of  $j^{\rm qt}(f)$ $\Leftrightarrow$ $\{ \upvarepsilon_{i}'\}$  produces a limit point of $j^{\rm qt}(M(f))$
and that these limit points coincide.
 \end{proof}

\section{Convergence}

For $f=f_{0}\in k_{\infty}-k$
let $a_{0}\in A$ be its polynomial part.  Then 
writing $f=a_{0}+ 1/f_{1}$, we let $a_{1}\in A$ be the polynomial part of $f_{1}$, and so forth:
 in this
way we obtain the {\it continued fraction expansion} of $f$:
\[   f = a_{0} + \frac{1}{f_{1}} =a_{0} + \frac{1}{ a_{1} + f_{2}}=\cdots .\]
The resulting series of iterated fractions 
\[  a_{0}, \;\; a_{0} + \frac{1}{ a_{1} }, \;\;  a_{0} + \frac{1}{ a_{1} +\frac{1}{a_{2}} },\;\; \dots \]
converges to $f$ and is infinite unless $f\in k$. 
Note that for all $i\geq 1$, $|1/f_{i}|>1$ and therefore $|a_{i}|>1$.  In particular, $a_{i}\not\in \F_{q}$ for all $i\geq 1$.
The continued fraction expansion is denoted
\[ f=[a_{0}, a_{1},\dots ] .\] 
Every  $f\in k_{\infty}-k$ may be so represented, and the representation is unique.
Moreover, $f\in k_{\infty}-k$ is quadratic if and only if its continued
fraction expansion is eventually periodic.  See \cite{Thak}, sect. 9.2.

The sequence of {\it denominators of best approximations} is defined recursively by 
\[ {\tt q}_{0}=1, {\tt q}_{1} =a_{1}, \dots , {\tt q}_{i} =a_{i} {\tt q}_{i-1} +  {\tt q}_{i-2}.\]
Note that 
\[ |{\tt q}_{n} | =|a_{1}\cdots a_{n}|= q^{\sum_{i=1}^{n}\deg(a_{i})} .\]
The corresponding sequence of duals $\{ {\tt q}_{n}^{\perp}\}$ is obtained by the recursion
\[{\tt q}^{\perp}_{0}=a_{0}, {\tt q}^{\perp}_{1} =a_{1}a_{0} +1, \dots , {\tt q}^{\perp}_{i} =a_{i} {\tt q}^{\perp}_{i-1} +  {\tt q}^{\perp}_{i-2}.\]  
Indeed, by \cite{Thak}, page 307:
\[ \| {\tt q}_{n}f\|=| {\tt q}_{n}f -{\tt q}^{\perp}_{n}| = \frac{1}{|{\tt q}_{n}||a_{n+1}|} =\frac{1}{|{\tt q}_{n+1}|}= q^{-\sum_{i=1}^{n+1}\deg(a_{i})}. \]

Since we will want to consider sums over monic polynomials, let us denote by
\[ \bar{\tt q}_{n} = c_{n}{\tt q}_{n}\]
the unique scalar multiple of ${\tt q}_{n}$ which is monic.  Clearly $|\bar{\tt q}_{n} |=|{\tt q}_{n} |$, $\bar{\tt q}_{n} ^{\perp}=c_{n}{\tt q}_{n}^{\perp}$
and $\| \bar{\tt q}_{n}f\| =\| {\tt q}_{n}f\|$.
It is a consequence of the above remarks that the following set forms a basis of $A$ as an $\F_{q}$-vector space:
\begin{align}\label{GenTpowbasis} 
\left\{   T^{\deg (a_{1})-1},\dots , 1; \; T^{\deg (a_{2})-1} \bar{\tt q}_{1},\dots , \bar{\tt q}_{1}; \; T^{\deg (a_{3})-1} \bar{\tt q}_{2},\dots ,
\bar{\tt q}_{2}; \; \dots  \right\} .
\end{align}
(The order in which the basis elements have been listed corresponds to decreasing errors, see (\ref{errormixedterm}) below.)
Using the basis (\ref{GenTpowbasis}), a typical element of $A$ can then be concisely written in the form
\begin{align}\label{conciseform}  a=c_{m}(T)\bar{\tt q}_{m} +c_{m+1}(T)\bar{\tt q}_{m+1}+ \cdots + c_{m'}(T) \bar{\tt q}_{m'},\quad 0\leq m\leq m', \end{align}
where $c_{i}(T)\in A$ is a polynomial of degree $\leq \deg (a_{i+1})-1$.  Notice that $a$ is monic
if and only if $c_{m'}(T) $ is monic.

\begin{lemm}\label{LambdaLemma}  
Let
$\upvarepsilon = q^{l-\sum_{i=1}^{N+1} \deg (a_{i})}$ for $0<l \leq \deg(a_{N+1})$. Then
\begin{align*} \Uplambda_{\upvarepsilon}(f ) =
{\rm span}_{\F_{q}} (T^{l-1}\bar{\tt q}_{N}, \dots , T\bar{\tt q}_{N}, \bar{\tt q}_{N};\; T^{\deg (a_{N+2})-1}\bar{\tt q}_{N+1}, \dots ).
\end{align*}
\end{lemm}

\begin{proof}  A general element $T^{r}\bar{\tt q}_{n}$ in the basis displayed in (\ref{GenTpowbasis}), where $0\leq r\leq \deg (a_{n+1})-1$, produces the error
\begin{align}\label{errormixedterm}  \| T^{r}\bar{\tt q}_{n}f\| = q^{r-\sum_{i=1}^{n+1} \deg (a_{i})}. \end{align}
Indeed we have $(T^{r}\bar{\tt q}_{n})^{\perp}=T^{r}\bar{\tt q}_{n}^{\perp}$ and 
\[  \| T^{r}\bar{\tt q}_{n}f\| =|T^{r}\bar{\tt q}_{n}f - T^{r}\bar{\tt q}_{n}^{\perp}|=|T^{r}|\|\bar{\tt q}_{n}f\|=q^{r-\sum_{i=1}^{n+1} \deg (a_{i})}. \]
From this the inclusion $\supset$ follows.  In view of the strictly decreasing pattern of errors of elements of the basis 
(\ref{GenTpowbasis}) and the non archimedean nature of $|\cdot |$, no other
basis elements may be involved in an element of $\Uplambda_{\upvarepsilon}(f)$, giving the equality
claimed.
\end{proof}

Denote $\upbeta_{N+1}:=\bar{\tt q}_{N+1}/\bar{\tt q}_{N}$.

 \begin{theo}\label{convergencetheo}  Let $f\in k_{\infty}-k$, $\upvarepsilon = q^{l-\sum_{i=1}^{N+1}\deg (a_{i})}$.  Then
\[ |j_{\upvarepsilon}(f)| = \left\{ \begin{array}{ll}
q^{q^{2}} & \text{if $l=1$ and $|\upbeta_{N+1}|>q$} \\
q^{q^{2}+q-1} & \text{otherwise} .
\end{array}
\right.\]
   \end{theo}
   
   The proof of Theorem \ref{convergencetheo} will follow immediately from two Lemmas, which we now present.
By Lemma \ref{LambdaLemma},
   \begin{align}\label{rescale} \frac{\Uplambda_{\upvarepsilon}(f)}{\bar{\tt q}_{N}} = {\rm span}_{\F_{q}} \{  1,\upalpha_{1},\upalpha_{2} \dots \}:={\rm span}_{\F_{q}} \{  1,\dots , T^{l-1},\upbeta_{N+1}, \dots \} . \end{align}
   The function $j_{\upvarepsilon}(f)$ may be calculated using the vector space (\ref{rescale}) in place of $\Uplambda_{\upvarepsilon}(f)$.  In particular, we may calculate  $j_{\upvarepsilon}(f)$ using the rescaled zeta function
   \[ \frac{\upzeta_{f,\upvarepsilon}(q^{n}-1)}{(\bar{\tt q}_{N})^{1-q^{n}}}=1+\sum_{i=1}^{\infty}\Upomega_{i}(q^{n}-1),\]   where
  \[ \Upomega_{i}(q^{n}-1) := \sum (c_{0}+\cdots + c_{i-1}\upalpha_{i-1}+\upalpha_{i})^{1-q^{n}},\] and where the sum is over $c_{0},\dots ,c_{i-1}\in\F_{q}$.  

    \begin{lemm}\label{omegalemma}   When $i=1$, 
\[   \Upomega_{1}(q^{n}-1)  =\frac{\upalpha_{1}^{q^{n}}-\upalpha_{1}}{\prod_{c\in\F_{q}}(c+\upalpha_{1}^{q^{n}})}\quad\text{and}\quad 
 \left| \Upomega_{1}(q^{n}-1)\right| = |\upalpha_{1}|^{q^{n}(1-q)}.\]
More generally, for all $i\geq 1$,
$ \left|  \Upomega_{i}(q^{n}-1)\right|   
\leq |\upalpha_{i}|^{q^{n}(1-q)}$.
\end{lemm}

\begin{proof} 
Write $\upalpha=\upalpha_{1}$.  Then 
\begin{align}\label{lemmafraction} \Upomega_{1}(q^{n}-1)=
\sum_{c}\frac{c+\upalpha}{c+\upalpha^{q^{n}}}=
 \frac{\sum_{c} (c+\upalpha)\prod_{d\not=c}(d+\upalpha^{q^{n}})}{\prod_{c}(c+\upalpha^{q^{n}})} .
\end{align}
Denote by $s_{i}(c)$ the $i$th elementary symmetric function on $\F_{q}-\{ c\}$.  Thus, $s_{0}(c)=1$, $s_{1}(c)=\sum_{d\not=c}d$, etc. Then the numerator of (\ref{lemmafraction}) may be written as 
\begin{align*}
\sum_{j=0}^{q-1}\bigg(\sum_{c} (c+\upalpha)s_{q-1-j}(c)\bigg)\upalpha^{jq^{n}} .\\
\end{align*}
Note that there is no constant term, and the coefficient of $\upalpha$ is $\prod_{d\not=0}d=-1$.
Now
\[  \sum_{c}cs_{q-2}(c)= \sum_{c\not=0}cs_{q-2}(c)=\sum_{c\not=0} c\prod_{d\not=0,c}d =\sum_{c\not=0}c\cdot (- c^{-1})=(q-1)(-1)=1,\]
which is the coefficient of $\upalpha^{q^{n}}$.
Moreover, $\sum_{c}s_{q-2}(c)=  s_{q-2}(0)  -\sum_{c\not=0}c^{-1}=0$, so the $\upalpha^{q^{n}+1}$ term vanishes.  
For $i<q-2$, we claim that
\begin{align*}
\sum_{c} cs_{i}(c)=0= \sum_{c} s_{i}(c). \end{align*}
When $i=0$,  $s_{0}(c)=1$ for all $c$, 
the terms $\upalpha^{q^{n}(q-1)}$, $\upalpha^{q^{n}(q-1)+1}$ have coefficients $\sum_{c}c=\sum_{c}1=0$ and so vanish.
When $i=1$, we have $q>3$,
so $s_{1}(c)=-c$ and  \[  \sum_{c}s_{1}(c)=-\sum_{c}c=0=-\sum_{c}c^{2} =\sum_{c}cs_{1}(c)\]
since the sums occurring above are power sums over $\F_{q}$ of exponent $1,2<q-1$.
For general $i<q-2$, we have $q>i+2$ and
\[ \sum_{c} s_{i}(c) = \sum_{c} P(c)\]
where $P(X)$ is a polynomial over $\F_{q}$ of degree $i<q-2$.   Hence  $\sum_{c} P(c)= \sum_{c} cP(c)=0$, since again, these are sums of powers of $c$ of exponent less than $q-1$.
Thus 
the numerator is $ \upalpha^{q^{n}}-\upalpha $
and the absolute value claim follows immediately.
When $i>1$, for each $\vec{c}$, let $\vec{c}_{+}=(c_{1},\dots, c_{i-1})$ and write 
$ \upalpha_{\vec{c}_{+}} = c_{1}\upalpha_{1} + \cdots + c_{i-1}\upalpha_{i-1} +\upalpha_{i}$.
Note trivially that $|\upalpha_{\vec{c}_{+}}|=|\upalpha_{i}|$.  Then by part (1),
\begin{align*}
 \left|  \Upomega_{i}(q^{n}-1)\right| = \left| \sum_{\vec{c}_{+}}\sum_{c} (c+\upalpha_{\vec{c}_{+}})^{1-q^{n}}\right|
 \leq \max \{ |\upalpha_{\vec{c}_{+}}|^{q^{n}(1-q)} \} =  |\upalpha_{i}|^{q^{n}(1-q)}.
\end{align*}
\end{proof}
The rescaling (\ref{rescale}) permits us to calculate
    \[  j_{\upvarepsilon}(f) =(T^{q}-T)^{q+2}\cdot \frac{\left( 1 +  \Upomega_{1}(q-1) + \cdots  \right)^{q+1}}{\Updelta},\quad \Updelta := U-V, \]
where 
   \begin{align*}
U & = \left( (T^{q}-T)(1+\Upomega_{1}(q-1)+\cdots ) \right)^{q+1} \\
& =T^{q(q+1)}-T^{q^{2}+1}-T^{2q}+T^{q+1}+T^{q(q+1)}\sum_{c} (c+\upalpha_{1})^{1-q} + \text{lower}
\end{align*}
(in the above, we use that $(T^{q}-T)^{q+1}=(T^{q^{2}}-T^{q})(T^{q}-T)$) and 
\begin{align*} V &= (T^{q}-T)(T^{q^{2}} -T)(1+\Upomega_{1}(q^{2}-1)+\cdots  ) \\
& = T^{q(q+1)}-T^{q^{2}+1}-T^{q+1}+T^{2}+T^{q(q+1)}\sum_{c} (c+\upalpha_{1})^{1-q^{2}} + \text{lower}.
\end{align*}
Thus \[ \Updelta= -T^{2q}+2T^{q+1}-T^{2} +T^{q(q+1)}\sum_{c} (c+\upalpha_{1})^{1-q} + \text{lower}.\]
Notice that we have made use of the estimates in Lemma \ref{omegalemma} to write ``$+ \text{lower}$'' in the above lines.

\begin{lemm}\label{DeltaLemma} $
|\Updelta| =\left\{ 
\begin{array}{ll}
q^{q+1} & \text{if  $|\upalpha_{1}|=q$} \\
q^{2q} & \text{otherwise}
 \end{array}
\right. 
$
\end{lemm}
 \begin{proof}  Suppose first that $|\upalpha_{1}|=q$.  Then $\upalpha_{1}=T$ or $\bar{\tt q}_{N+1}/\bar{\tt q}_{N}$ where the (unnormalized) best approximations
 satisfy ${\tt q}_{N+1}=a_{N+1}{\tt q}_{N}+{\tt q}_{N-1}$ for $a_{N+1}$ linear.  Since the normalized best approximations are monic, it follows that we may write $\upalpha_{1}=T+\updelta$
 where $|\updelta |<q$.  In particular, by Lemma  \ref{omegalemma}, item (1), we have
 \begin{align*} \left| \sum_{c} (c+\upalpha_{1})^{1-q} -\sum_{c} (c+T)^{1-q} \right| & = \left|\frac{(\upalpha_{1}^{q}-\upalpha_{1})\prod_{c}(c+T^{q})  
 -(T^{q}-T)\prod_{c}(c+\upalpha_{1}^{q})}{\prod_{c} \big( (c+\upalpha_{1}^{q})(c+T^{q})\big)}\right| \\
 & < q^{q(1-q)}.   
 \end{align*}
 Therefore, we may write 
 \[ \Updelta= -T^{2q}+2T^{q+1}+T^{q(q+1)}\sum_{c} (c+T)^{1-q} + \text{lower}.\]
 By Lemma \ref{omegalemma}, item (1),
 \begin{align*}
  -T^{2q} +2T^{q+1}+ T^{q(q+1)}\sum_{c} (c+T)^{1-q}  & =  \\ 
  \frac{(-T^{2q}+2T^{q+1} )\cdot \prod_{c} (c+T^{q})+T^{q(q+1)}\cdot (T^{q}-T)}{\prod_{c} (c+T^{q})} & = \\
  \frac{T^{q(q+1)+1}+ \text{lower} }{\prod_{c} (c+T^{q})}.
  \end{align*}
It follows that,
$\left| -T^{2q}+2T^{q+1}+T^{q(q+1)}\sum (c+T)^{1-q} \right|=q^{q+1} ,
$
 and we conclude that in this case,
$ | \Updelta|=q^{q+1}$.  
  If $|\upalpha_{1}|>q$, by Lemma \ref{omegalemma}, item (1),
 \[  \left|T^{q(q+1)}\sum (c+\upalpha_{1})^{1-q}\right|=q^{q(q+1)}\cdot |\upalpha_{1}|^{q(1-q)}<q^{2q}=|T^{2q}| \]
 hence $ | \Updelta|=q^{2q}$.
\end{proof}

\begin{proof}[Proof of Theorem 3]When $l\not=1$, then $\upalpha_{1}=T$: thus $|\upalpha_{1}|=q$, $|\Updelta|=q^{q+1}$ and $|j_{\upvarepsilon}(f)|=q^{q^{2}+q-1}$.  The same is true when $l=1$
and $\upbeta_{N+1}=\upalpha_{1}$ satisfies $|\upbeta_{N+1}|=q$.  If $l=1$ and $|\upbeta_{N+1}|=|\upalpha_{1}|>q$, then $|\Updelta|=q^{2q}$ and  $|j_{\upvarepsilon}(f)|=q^{q^{2}}$.  
\end{proof}

Combining Theorems \ref{rationaltheorem} and \ref{convergencetheo}, we arrive at the 
\begin{coro} $f\in k$ $\Leftrightarrow$ $ j_{\upvarepsilon}(f)=\infty$ for $\upvarepsilon$ sufficiently small $\Leftrightarrow$
$\infty\in j(f)$.
\end{coro}


\begin{thebibliography}{00}


\bibitem [1] {Cas}  Cassels, J.W.S., {\it An Introduction to Diophantine Approximation}. Cambridge Tracts in Mathematics and Mathematical Physics, {\bf 45}. Cambridge University Press, New York, 1957. 
\bibitem [2] {DGII} Demangos, L \& Gendron, T.M., Quantum $j$-Invariant in Positive Characteristic II: Formulas and Values at the Quadratics. 
\bibitem [3]{Ge-C}  Casta\~{n}o Bernard, C. \& Gendron, T.M., Modular invariant of quantum tori.   {\it Proc. Lond. Math. Soc.} {\bf 109} (2014), Issue 4, 1014--1049. 
\bibitem [4]{Gek} E. U. Gekeler, Zur Arithmetik von Drinfeld Moduln, Math. Ann. {\bf 262} (1983), 167--182. 
\bibitem [5]{Gek1} E. U. Gekeler, On the coefficients of Drinfeld modular forms, Invent. Math. {\bf 93} (1988), Issue 3, 667--700.
\bibitem [6]{Goss} D. Goss, {\it Basic structures of Function Field Arithmetic}, Springer-Verlag, Berlin, 1998.
\bibitem [7]{KN} Kuipers, L \& Niederreiter,H., {\it Uniform Distribution of Sequences.} Dover, N.Y., 2012.
\bibitem [8]{Thak} Thakur, D.S., {\it Function Field Arithmetic}, World Scientific, Singapore, 2004.
\bibitem [9]{VS} Villa Salvador, Gabriel Daniel, {\it Topics in the Theory of Algebraic Function Fields}, Birkh\"{a}user, Boston, 2006.
\bibitem [10]{Weil} Weil, Andr\'{e}, {\it Basic Number Theory}, Springer-Verlag, Berlin, 1974.
\end{thebibliography}
 \end{document}